\documentclass{rspublica}

\usepackage{amssymb}
\usepackage{amsmath}

\newcommand{\R}{{\Bbb R}}

\newcommand{\N}{{\Bbb N}}
\newcommand{\C}{{\Bbb C}}

\title[Uniqueness of fast travelling fronts]{Uniqueness of fast travelling fronts in
reaction-diffusion equations with delay}
\author[Aguerrea, Trofimchuk and Valenzuela]{
{Maitere Aguerrea,} {Sergei Trofimchuk$^*$\footnotetext{$^*$Author
for correspondence (trofimch@inst-mat.utalca.cl)}} {and Gabriel
Valenzuela}} \affiliation{Instituto de Matem\'atica y Fisica,
Universidad de Talca, Casilla 747, Talca, Chile  } \label{firstpage}

\begin{document}

\maketitle

\label{firstpage} \maketitle
\begin{abstract}{Time-delayed reaction-diffusion equation; monostable
case; uniqueness; travelling front; single species population model.
} We consider positive travelling fronts $u(t,x) = \phi(\nu \cdot x
+ct), \ \phi(-\infty) =0,\ \phi(\infty) =\kappa$, of the equation
$u_t(t,x) = \Delta u(t,x)- u(t,x) + g(u(t-h,x)), \ x \in \R^m \ (*)
$. It is assumed that $(*)$ has exactly two non-negative equilibria:
$u_1 \equiv 0$ and $u_2 \equiv \kappa>0$. The birth function $g \in
C^2(\R, \R)$ may be non-monotone on $[0,\kappa]$. Hence, we are
concerned with the so-called monostable case of the time-delayed
reaction-diffusion equation.  Our main result says that for every
fixed and sufficiently large velocity $c$, the positive travelling
front $\phi(\nu \cdot x +ct)$ is unique (modulo translations).
Notice that $\phi$ can be non-monotone.  To prove the uniqueness, we
introduce a small parameter $\varepsilon = 1/c$ and realize the
Lyapunov-Schmidt reduction in a scale of Banach spaces.
\end{abstract}

\section{Introduction and main result}
\label{sec.int} \noindent  In this paper, we consider the
time-delayed reaction-diffusion equation
\begin{equation}\label{17}
u_t(t,x) = \Delta u(t,x)  - u(t,x) + g(u(t-h,x)), \ u(t,x) \geq 0,\
x \in \R^m.
\end{equation}
Eq. (\ref{17}) and its non-local versions are used widely to model
many physical, chemical, ecological and biological processes, see
Faria {\it et al.} (2006) for more references. The nonlinearity
$g$  is referred to in the ecology literature as the {\it birth
function}, and we will suppose that $-x+g(x)$  is of the
monostable type. Thus Eq. (\ref{17}) has exactly two non-negative
equilibria $u_1 \equiv 0, \ u_2 \equiv \kappa >0$. We say that the
wave solution $u(x,t) = \phi(\nu \cdot x +ct),$ $ \|\nu\| =1,$ of
(\ref{17}) is a wavefront (or a travelling front), if the profile
function $\phi$ satisfies the boundary conditions $\phi(-\infty) =
0$ and $\phi(+\infty) = \kappa$. After scaling, such a profile
$\phi$ is a positive heteroclinic solution of the delay
differential equation
\begin{equation}\label{twe0}
\varepsilon^2 x''(t) - x'(t)-x(t)+ g(x(t-h))=0,\ \varepsilon := 1/c
>0, \quad t \in \R.
\end{equation}
Notice that $\phi$ may not be monotone. Since the biological
interpretation of $u$ is the size of an adult population, we will
consider  only positive travelling fronts.

If we take $h = 0$ in (\ref{17}), we obtain a monostable
reaction-diffusion equations without delay. The problem of
existence of travelling fronts for this equation is quite well
understood. In particular, for each such equation we can indicate
a positive real number $c_*$ such that, for every $c \geq c_*$, it
has exactly one travelling front $u(x,t) = \phi(\nu \cdot x +ct)$.
Furthermore, Eq. (\ref{17}) does not have any travelling front
propagating at the velocity $c < c_*$. The profile $\phi$ is
necessarily strictly increasing function. See, for example,
Theorem 8.3 (ii), Theorem 8.7 and Theorem 2.39 in Gilding \&
Kersner (2004).

However, the situation will change drastically if we take $h
>0$. Actually, at the present moment, it seems that we are far from
proving similar results concerning  the existence, uniqueness and
geometric properties of wavefronts for delayed equation
(\ref{17}). This despite that fact that the existence of
travelling fronts in (\ref{17}) was recently intensively studied
for some specific subclasses of birth functions. E.g. see So, Wu
\& Zou (2001), Wu \& Zou (2001), Faria {\it et al.} (2006), Ma
(2007), Trofimchuk \& Trofimchuk (2008) and references wherein.
Certainly, so called monotone case (when $g$ is monotone on
$[0,\kappa]$) is that one for which the most information is
available.  But so far, even for equations with monotone birth
functions very little is known about the number of wavefronts
(modulo translation) for an arbitrary fixed $c\geq c_*$.  In
effect, there exist a very few theoretical studies devoted to the
uniqueness problem for equations similar to (\ref{17}). To the
best of our knowledge, the uniqueness was established only for
small delays in Ai (2007) and for a family of unimodal and
piece-wise linear birth functions in Trofimchuk {\it et al.}
(2007). The mentioned family is rather representative since {\it
'asymmetric'} tent maps mimic the main features of general
unimodal birth functions. In fact, we believe that the uniqueness
of positive wavefront can be proved for delayed equations with the
unimodal birth function satisfying the following assumptions:
\begin{description}
\item[{\rm \bf(H)}] The steady state $x_1(t) \equiv \kappa >0$ (respectively $x_2(t)\equiv
0$) of  the equation
\begin{equation}\label{17a}
x'(t) =   - x(t) + g(x(t-h))
\end{equation}
is exponentially  stable and globally attractive (respectively
hyperbolic).
\end{description}
\begin{description}
\item[{\rm \bf(G)}]
$g \in C^1(\R_+, \R_+), \ p:=g'(0)>1,$ and $g''(x)$ exists and is
bounded near $0$. We suppose that $g$ has exactly two fixed points
$0$ and $\kappa>0$. Set $A = \sup\{a \in(0,\kappa/2]: g'(x)
>0 , \ x \in [0,a) \}$ and $\zeta_2 = \max
_{x \in [0,\kappa]}g(x)$, we assume  that $g(x) > 0$ for $x \in
(0, \zeta_2]$. Then there exists  a positive $\zeta_1 \leq \min
\{g (\zeta_2), A\}$ such that $g(\zeta_1) = \min_{s \in
[\zeta_1,\zeta_2]}g(s)$. Notice that
$g([\zeta_1,\zeta_2])\subseteq [\zeta_1,\zeta_2]$.  Without
restricting the generality, we can suppose that $\sup_{s \geq 0}
g(s) \leq \zeta_2$.
\end{description}
In this paper, we follow the approach of Faria {\it et al.} (2006)
to prove the uniqueness (up to translations) of positive wavefront
for a given fast speed $c$. In the case of (\ref{17}), this
approach essentially relies on the fact that, in 'good' spaces and
with suitable $g'(0),$ $g'(\kappa)$, the linear operator
$(\mathcal{L}x)(t)= x'(t)+x(t)-g'(\psi(t-h))x(t-h)$ is a
surjective Fredholm operator. Here $\psi$  is a heteroclinic
solution of equation (\ref{twe0}) considered with $\varepsilon=0$.
In consequence, the Lyapunov-Schmidt reduction can be used to
prove the existence of a smooth family of travelling fronts in
some neighborhood of $\psi$. As it was shown in Faria \&
Trofimchuk (2006)  this family contains positive solutions as
well. However, an important and natural question about the number
of the positive wavefronts has not been answered in the past. We
solve this problem in the present paper, establishing the
following result:
\begin{theorem} \label{mr}  Assume {\rm \bf(H)}, {\rm \bf(G)}. Then there
exists a unique (modulo translations) positive wavefront of Eq.
(\ref{17}) for each sufficiently large speed $c$.
\end{theorem}
In order to apply Theorem \ref{mr}, one needs to find sufficient
conditions to ensure the global attractivity of the positive
equilibrium of (\ref{17a}). Some results in this direction were
found in Liz {\it et al.} (2005) for nonlinearities
 satisfying a
generalized Yorke condition. In particular, Corollary 2.3 of the
latter paper implies the following  \noindent
\begin{corollary} \label{nasa} Assume {\rm \bf(H)} and {\rm \bf(G)}, and that
either $\Gamma := g'(\kappa) \in [0,1]$ or \begin{equation*}
\label{gsc} \Gamma < 0 \quad {\rm and} \quad e^{- h} > - \Gamma \ln
\frac{{\Gamma}^{2} - \Gamma}{{\Gamma}^{2} + 1}.
\end{equation*}
Suppose  also that $g \in C^3(\R_+, \R_+)$ has only one critical
point $x_M$ (maximum) and  that the Schwarzian
$(Sg)(x)=g'''(x)(g'(x))^{-1}-(3/2) \left(g''(x)(g'(x))^{-1}\right)^2
$ is negative for all $x >0$, $x \not=x_M$. Then the conclusion of
Theorem \ref{mr} holds true.
\end{corollary}
Notice that Corollary \ref{nasa} applies to both the Nicholson's
blowflies equation and the Mackey-Glass equation with non-monotone
nonlinearity.

The structure of this paper is as follows: the next section
contains preliminary facts and fixes some notation. In the third
section, following  Faria {\it et al.} (2006), we realize the
Lyapunov-Schmidt reduction in a scale of Banach spaces. Section 4
contains the core lemma   of the paper. As an applications of this
lemma, we obtain an alternative proof of the existence of positive
wavefronts, see Theorem \ref{mrr13}. Finally, in Section 5 we show
that there exists exactly one wavefront for each fixed fast speed.
\section{Preliminaries}
This section contains several auxiliary results that will be
needed later. Proofs of them (excepting Lemma \ref{6}) can be
found in Liz {\it et al.} (2002) [Lemma \ref{13}], Trofimchuk {\it
et al.} (2007) [Lemma \ref{T2}], Faria \& Trofimchuk (2006)
[Lemmas \ref{char3}-\ref{8}].
\begin{lemma} \label{13} Assume $\mathbf{(G)}$.
If $x\not\equiv 0$ is a non negative solution of Eq. (\ref{17a}),
then
$$
\zeta_1  \leq \liminf\limits_{t \to +\infty} x(t) \leq
\limsup\limits_{t \to +\infty} x(t) \leq \zeta_2.
$$
\end{lemma}
\begin{lemma} \label{T2} Assume $\mathbf{(G)}$.
Consider wavefront $u(x,t) = \phi(\nu \cdot  x +ct),\  \|\nu\| =1,$
to Eq. (\ref{17}). Then there exists a unique $\tau$ such that
$\phi(\tau)= A, \ \phi'(s) > 0$ for all $s \leq \tau$.
\end{lemma}
\begin{lemma} \label{char3}
Suppose that $p> 1$ and  $h >0$. Then the characteristic equation
\begin{equation}\label{char}
z = - 1 + p\exp (-z h)
\end{equation}
has only one real root $0 <  \lambda < p-1$. Moreover, all roots
$\lambda, \lambda_j, \ j =2,3,\dots $ of (\ref{char}) are simple and
we can enumerate them in such a way that $\lambda > \Re \lambda_2 =
\Re \lambda_3 \geq \dots$
\end{lemma}
\noindent Everywhere in the sequel, $\lambda_j$ stands for a root
of (\ref{char}). Notice that we write $\lambda$ instead of
$\lambda_1$.
\begin{lemma} \label{psi} Assume {\rm \bf(H)}, {\rm \bf(G)} and let $\lambda$ be as in Lemma \ref{char3}.
Then (\ref{17a}) has a unique (modulo translations) positive
heteroclinic solution $\psi$. Moreover, $\psi(t-t_0) = \exp
(\lambda t) + O(\exp((2\lambda- \delta)t)), \ t \to - \infty,$ for
each $\delta
> 0$ and some $t_0 \in \R$.
\end{lemma}
\begin{lemma} \label{lemma}
Let $\{\lambda_{\alpha}(\varepsilon), \alpha \in A\}$, where $\N
\cup \{ \infty \}\subset A$, denote the (countable) set of roots to
the equation
\begin{equation}\label{char2}
\varepsilon^2z^2-z-1+p\exp(-z h)=0.
\end{equation}
If $p> 1, \ h >0, \ \varepsilon  \in (0,1/(2\sqrt{p-1}))$ then
(\ref{char2}) has exactly two real roots $\lambda_1(\varepsilon),
\lambda_{\infty}(\varepsilon)$ such that
$$0<\lambda<  \lambda_1(\varepsilon) < 2(p-1) <
\varepsilon^{-2}-2(p-1) < \lambda_{\infty}(\varepsilon) <
\varepsilon^{-2}+1.$$ Moreover: (i) there exists an interval
$\mathcal{O}=\mathcal{O}(p,h) \ni 0$ such that, for every
$\varepsilon \in \mathcal{O}$, all roots
$\lambda_{\alpha}(\varepsilon), \alpha \in A$ of (\ref{char2}) are
simple  and the functions $\lambda_{\alpha}: \mathcal{O} \to \C$
are continuous; (ii) we can enumerate $\lambda_j(\varepsilon), j
\in \N$, in such a way that there exists $\lim_{\varepsilon \to
0+}\lambda_j(\varepsilon) =\lambda_j$ for each $j \in \N$, where
$\lambda_j \in \C$ are the roots of (\ref{char}), with
$\lambda_1=\lambda$; (iii) for all sufficiently small
$\varepsilon$, every vertical strip $\xi \leq \Re z \leq 2(p-1)$
contains only a finite set of $m(\xi)$ roots (if $\xi \not\in
\{\Re \lambda_j, \ j \in \N\}$, then $m(\xi)$ does not depend on
$\varepsilon$) $\lambda_1(\varepsilon), \dots,
\lambda_{m(\xi)}(\varepsilon)$ to (\ref{char2}), while the
half-plane $\Re z > 2(p-1)$ contains only the root $
\lambda_\infty(\varepsilon)$.
\end{lemma}
\noindent Assume {\rm \bf(H)}, {\rm \bf(G)}, and let  $\psi$  be
the positive heteroclinic solution from Lemma \ref{psi}. For a
fixed $\mu\geq  0$,  we set $ \#\{\lambda_j: \mu < \Re{\lambda_j
\}}: = d(\mu)$ and  $\|x\|^+ = \sup_{\R_+}|x(s)|$, $\|x\|^-_{\mu}
= \sup_{\R_-}e^{-\mu s}|x(s)|$, $|x|_{\mu} =
\max\{\|x\|^+,\|x\|^-_{\mu}\}$. Consider the  Banach space
$$
C_{\mu}(\R) = \{x \in C(\R,\R): \|x\|^-_{\mu} < \infty, \ x(-\infty)
=0, \ {\rm \ and \ } x(+\infty) \ {\rm is \ finite } \},
$$ equipped with the norm $|x|_{\mu}$. We will need the operators $\mathcal{G},
\mathcal{I}, \mathcal{I}_{\varepsilon},
\mathcal{I}^+_{\varepsilon}, \mathcal{I}^-_{\varepsilon},
\mathcal{N}: C_{\mu}(\R) \to C_{\mu}(\R)$, where $
(\mathcal{G}x)(t) = g(x(t))$ is the Nemitski operator, $
\mathcal{I}= \mathcal{I}^-_{0}, \mathcal{I}^+_{0}=0$, $
\mathcal{I}_{\varepsilon}
=\sigma^{-1}(\varepsilon)(\mathcal{I}^+_{\varepsilon}+
\mathcal{I}^-_{\varepsilon}), \ \sigma(\varepsilon) :=
\sqrt{1+4\varepsilon^2},$  and
$$ (\mathcal{I}^+_{\varepsilon}x)(t)=
 \int^{+\infty}_te^{\frac{(1+
\sigma(\varepsilon))(t-s)}{2\varepsilon^2}}x(s-h)ds,  \
(\mathcal{I}^-_{\varepsilon}x)(t)=
\int_{-\infty}^te^{\frac{-2(t-s)}{1+ \sigma(\varepsilon)}}x(s-h)ds,
$$
$$  \ (\mathcal{N}x)(t) =
\int^t_{-\infty}e^{-(t-s)}q(s)x(s-h)ds, \ q(s) : = g'(\psi(s-h)).
$$
Since $g'(x) = p+ O(x), \ x\to 0,$ and $\psi(t) = O(\exp(\lambda
t)), \ t \to -\infty$, we obtain that
$$
q(t) = p + \epsilon(t), \ \epsilon(t)=O(\exp(\lambda t)), \ t \to
-\infty;  \ {\rm and }\  q(-\infty) =  p > 1,\ q(\infty) =
g'(\kappa).$$ Observe that $\mathcal{I}^{\pm}_{\varepsilon},
\mathcal{N}$ are well defined: e.g. $(\mathcal{N}x)(+\infty) =
g'(\kappa)x(+\infty)$ and, for $t \leq h$,
$$
|(\mathcal{N}x)(t)| \leq
\int^t_{-\infty}e^{-(t-s)}|q(s)|\|x\|^-_{\mu}e^{\mu(s-h)}ds \leq
\frac{\|x\|^-_{\mu}\sup_{t \le h}|q(t)|}{1+\mu}e^{\mu (t-h)}.
$$
\begin{lemma}  \label{8}
Operator families $\mathcal{I}^{\pm}_{\varepsilon}:
(-1/\sqrt{\mu}, 1/\sqrt{\mu}) \to \mathcal{L}(C_{\mu}(\R)),\ \mu
\geq 0,$ are continuous in the operator norm. In particular,
$\mathcal{I}_{\varepsilon} \to \mathcal{I}$ as $\varepsilon \to
0$.
\end{lemma}
\begin{lemma} \label{6} If  {\rm \bf(H)} holds and $\mu \not\in\{\Re\lambda_j\}, \ \mu \geq 0$, then $I - \mathcal{N}:
C_{\mu}(\R) \to C_{\mu}(\R)$ is a surjective Fredholm operator and
$\dim$ {\rm Ker} $(I - \mathcal{N}) = d(\mu)$.
\end{lemma}
\begin{proof}
First, we establish that $I - \mathcal{N}$ is an epimorphism. Take
some $d \in C_{\mu}(\R)$ and consider the following integral
equation
$$
x(t) - \int^t_{-\infty}e^{-(t-s)}q(s)x(s-h)ds = d(t).
$$
If we set $z(t) = x(t)-d(t)$, this equation is transformed into
$$
z(t) - \int^t_{-\infty}e^{-(t-s)}q(s)(z(s-h)+d(s-h))ds = 0.
$$
Hence, in order to establish the surjectivity of $I - \mathcal{N}$,
it suffices  to prove the existence of $C_{\mu}(\R)$-solution of the
equation
\begin{equation}\label{bso}
z'(t) = - z(t) + q(t)z(t-h) +q(t)d(t-h).
\end{equation}
First, notice that all solutions of (\ref{bso}) are bounded on the
positive semi-axis $\R_+$ due to the boundedness of $q(t)d(t-h)$ and
the exponential stability of the homogeneous $\omega$-limit equation
$z'(t) = - z(t) + g'(\kappa)z(t-h).$  Here we use the persistence of
exponential stability under small bounded perturbations (e.g. see
Section 5.2 in Chicone \& Latushkin (1999)) and the fact that
$q(+\infty) = g'(\kappa)$. Furthermore, since every solution $z$ of
(\ref{bso}) satisfies $z'(t) = -z(t) + g'(\kappa)z(t-h) +
g'(\kappa)d(+\infty) + \epsilon(t)$ with $\epsilon(+\infty) = 0$, we
get $z(+\infty) =d(+\infty) g'(\kappa)(1-g'(\kappa))^{-1}$.  Next,
by effecting the change of variables $z(t)=\exp (\mu t)y(t)$ to Eq.
(\ref{bso}), we get a linear inhomogeneous equation of the form
\begin{equation}\label{bso0}
y'(t) = -(1+\mu) y(t) + [p\exp(-\mu h) + \epsilon_{1}(t)]y(t-h) +
\epsilon_{2,\mu}(t),
\end{equation}
where $\epsilon_1(-\infty) = \epsilon_{2,0}(-\infty) =0$ and
$\epsilon_{2,\mu}(t) = O(1), \ \mu >0,$ at $t = - \infty$. Since
the $\alpha$-limit equation $y'(t) = -(1+\mu) y(t) + p\exp(-\mu
h)y(t-h), \mu \not\in\{\Re\lambda_j\},$ to the homogeneous part of
(\ref{bso0}) is hyperbolic, due to the above mentioned persistence
of the property of exponential dichotomy, we again conclude that
Eq. (\ref{bso0}) also has an exponential dichotomy on $\R_-$. Thus
(\ref{bso0}) has a solution $y^*_{\mu}$ which is bounded on
$\R_{-}$ (while $y^*_{0}(-\infty) =0$) so that $z^*(t) = \exp (\mu
t)y^*_{\mu}(t) = O(\exp(\mu t)), \ t \to -\infty,$ is a
$C_{\mu}(\R)$-solution of Eq. (\ref{bso}).

Next we prove that $\dim$ {\rm Ker}$(I - \mathcal{N}) =
\#\{\lambda_j: \mu  < \Re\lambda_j \}$. It is clear that $\phi_{j}
\in {\rm Ker}(I - \mathcal{N})$ if and only if $\phi_{j}$ is a
$C_{\mu}(\R)-$solution of the  equation
\begin{equation}\label{veq}
\phi'(t) = - \phi(t) + q(t)\phi(t-h).
\end{equation}
We already have seen that every solution of (\ref{veq}) satisfies
$\phi(+\infty)=0$, thus we only have to show that there exist
solutions $\phi_j$ with $\|\phi_j\|^-_{\mu}< \infty$. In fact, we
will prove that for each $\Re \lambda_{j}
> \mu$ and $\delta \in (0, \min\limits_{\Re\lambda_j >0,\ \lambda
> \Re\lambda_i
>0}\{\Re\lambda_j, \lambda - \Re\lambda_i\})$ there is $\phi_{j}(t)=e^{\lambda_{j}
t}+e^{\sigma t}v_{j}(t)\in {\rm Ker}(I - \mathcal{N}),$ with
$\sigma=\lambda+\delta, \ v_{j}(t)=O(1), \ t \to -\infty.$ Set
$q(t) = p + \epsilon(t)$, then
 $v_{j}(t)$ can be chosen as a bounded solution of the
equation
\begin{equation}\label{av}
v'(t) + (1+\sigma)v(t)-(p+\epsilon(t))e^{-\sigma h}v(t-h)=
e^{-\lambda_{j} h+(\lambda_{j}-\sigma )t}\epsilon(t).
\end{equation}
Since  $e^{-\lambda_{j} h+(\lambda_{j}-\sigma )t}\epsilon(t) =
O(e^{(\Re \lambda_j-\delta) t})$ at $-\infty$,  we get the
following $\alpha$-limit form  of (\ref{av})
\begin{equation*}\label{aav}
v'(t) + (1+\sigma)v(t)-pe^{-\sigma h}v(t-h)= 0.
\end{equation*}
This autonomous equation is exponentially stable since its
characteristic equation
$$
z+\lambda+\delta=-1+pe^{-(z+\lambda+\delta)h}
$$
has roots $z_{j}=\lambda_{j}-\lambda-\delta$ with $\Re
z_{j}=\Re\lambda_{j}-\lambda-\delta<0$. Thus (\ref{av}) has a
unique solution $v_{j}$ bounded in $\R_{-}$. Is clear that
$d(\mu)$ solutions  $\{\phi_{j}\}$ are linearly independent, we
claim that, in fact,  system $\{\phi_{j}\}$ generates {\rm Ker}$(I
- \mathcal{N})$. Indeed,
 suppose for an instance that $\phi \in $ {\rm Ker}$(I
- \mathcal{N})- <\phi_{j}>$.

As $\phi$ solves the equation
$$
x'(t) = - x(t) + px(t-h)+O(\exp((\lambda+\mu) t)), \ t \to -\infty,
$$
we get (e.g. p. 28 in Mallet-Paret (1999))
$$
\phi(t)=z(t)+O(\exp{((\lambda+\mu-\delta})t)), \ t \to -\infty,
$$
where $z(t)$ is the eigensolution corresponding to the eigenvalues
$\zeta$ with $\mu\leq\Re\zeta<\lambda+\mu $. In this way,
\begin{equation}\label{phii}
\phi(t)=C\exp{(\lambda t)}+\sum_{j=2}^{d(\mu)}C_{j}\exp{(\lambda_{j}
t)}+O(\exp{((\lambda+\mu-\delta})t)), \ t \to -\infty.
\end{equation}
Now take
$$w(t)=C(\exp{(\lambda t)}+\exp{(\sigma t)}v_1(t))+\sum_{j=2}^{d(\mu)}C_{j}(\exp{(\lambda_{j} t)}+\exp{(\sigma t)}v_{j}(t)) \in <\phi_j>.$$
Since $\exp{(\sigma t)}v_{j}(t)=O(\exp{(\lambda+\delta)t}), \ t \to
-\infty,$ we can write
\begin{equation*}
w(t)=C\exp{(\lambda t)}+\sum_{j=2}^{d(\mu)}C_{j}\exp{(\lambda_{j}
t)}+O(\exp{((\lambda+\delta})t)), \ t \to -\infty.
\end{equation*}
Thus $\Delta(t): =\phi(t)-w(t)$ satisfies $\Delta(t)
=O(\exp{(\lambda-\delta})t), \ t \to -\infty,$ and solves
\begin{equation}\label{pw}
x'(t)=-x(t)+px(t-h)+O(\exp{((2\lambda-\delta)t)}), \ t \to
-\infty.
\end{equation}
Applying Proposition 7.1 from Mallet-Paret (1999) we conclude that
$$
\Delta(t)=z(t)+O(\exp{((2\lambda-\delta-\delta/2)t)}), \ t \to
-\infty,
$$
where $z(t)$ is the eigensolution corresponding to the eigenvalues
$\zeta$ such that $\lambda-\delta \leq \Re\zeta<
2\lambda-\delta$ and in consequence $z(t)=C_{1}e^{\lambda t}$, for some $C_{1}$.
Hence,
$$
\phi(t) = w(t) + \Delta(t) = C'\exp{(\lambda
t)}+\sum_{j=2}^{d(\mu)}C_{j}\exp{(\lambda_{j}
t)}+O(\exp{((\lambda+\delta})t)), \ t \to -\infty,
$$
for small $\delta >0$. The latter formula improves (\ref{phii}), and
if we take
$$w_1(t)=C'(\exp{(\lambda t)}+\exp{(\sigma t)}v_1(t))+\sum_{j=2}^{d(\mu)}C_{j}(\exp{(\lambda_{j} t)}+\exp{(\sigma t)}v_{j}(t)) \in <\phi_j>,$$
then  $ \Delta_1(t)= \phi(t)- w_1(t)=O(\exp{(\lambda +\delta) t)},
\ t \to -\infty$.  Since $\Delta_1(t)$ satisfies
\begin{equation*}
x'(t)=-x(t)+px(t-h)+O(\exp{((2\lambda+\delta)t)}), \ t \to
-\infty,
\end{equation*}
we can proceed as before to get $
\Delta_1(t)=z_{1}(t)+O(\exp{(2\lambda +\delta-\delta/2) t)}, \ t
\to -\infty,$ where $z_{1}(t)$ is the eigensolution corresponding
to the eigenvalues $\zeta$ such that $\lambda+\delta \leq \zeta
<2\lambda +\delta$. Thus $z_1(t) = 0$ and
$\Delta_1(t)=O(\exp{(2\lambda +\delta-\delta/2) t)}, \ t \to
-\infty$. Iterating this procedure (and subtracting $\delta/2^{k}$
from the exponent $2\lambda +\delta$ on the step $k$), we can
conclude that $ \Delta_{1}(t)=O(\exp{(k\lambda t)}), \ t \to
-\infty, \ k\geq 2.$ This means that $\Delta$ is a small solution
of (\ref{veq}). However, Eq. (\ref{veq})  cannot have solutions
with superexponential decay at $-\infty$ (e.g see p. 9 in Faria \&
Trofimchuk (2006)) and thus $\Delta(t)=0$. This implies that $\phi
\in <\phi_j>$,  a contradiction.
\end{proof}

Throughout the rest of the paper, we will suppose that the
$C^1$-smooth function $g$ is defined and bounded on the whole real
axis $\R$. This assumption does not restrict the generality of our
framework, since it suffices to take any smooth and bounded
extension on $\R_-$ of the nonlinearity $g$ described in {\rm
\bf(G)}. Notice that, since there exists finite $g'(0)$, we have
$g(x) = x\gamma(x)$ for a bounded $\gamma \in C(\R)$. Set $
\gamma_0= \sup_{t \in \R}|\gamma(x)|$.  As it can be easily
checked, $|\mathcal{G}x|_\mu \leq \gamma_0|x|_\mu$ so that
actually $\mathcal{G}$ is well-defined. Furthermore, we have the
following lemma:
\begin{lemma}\label{7} Assume that $g \in C^1(\R)$. Then $\mathcal{G}$ is Fr\'echet continuously
differentiable on $C_{\mu}(\R)$ with differential
$\mathcal{G}'(x_0): y(\cdot) \to g'(x_0(\cdot))y(\cdot).$
\end{lemma}
\begin{proof} We have that $|\mathcal{G}'(x)u|_\mu=
|g'(x(\cdot))u(\cdot)|_\mu \leq \sup_{t \in \R} |g'(x(t))||u|_\mu$.
By the Taylor formula, $ g(v)-g(v_0) - g'(v_0)(v-v_0) = (
g'(\theta)- g'(v_0))(v-v_0),$ $ \ \theta \in [v,v_0]$. Fix some $x_0
\in C_{\mu}(\R)$. Since functions in $C_{\mu}(\R)$ are bounded and
$g'$ is uniformly continuous on bounded sets of $\R$, for any given
$\delta
>0$ there is $\sigma >0$ such that for $|x-x_0|_\mu<\sigma$ we have
that $ |\mathcal{G}x-\mathcal{G}x_0 - g'(x_0(\cdot))(x-x_0)|_\mu
\leq \delta |x-x_0|_\mu$ and $
\|\mathcal{G}'(x)-\mathcal{G}'(x_0)\|_{\mathcal{L}(C_{\mu}(\R))}
<\delta. $
\end{proof}
\section{Lyapunov-Shmidt reduction}
Being a bounded solution of Eq. (\ref{twe0}), each travelling wave
should satisfy
\begin{equation}\label{psea} \hspace{0mm}
x(t) = \frac{1}{\sigma(\varepsilon)}(
\int\limits_{-\infty}^te^{\frac{-2(t-s)}{1+
\sigma(\varepsilon)}}g(x(s-h))ds +
\int\limits^{+\infty}_te^{\frac{(1+
\sigma(\varepsilon))(t-s)}{2\varepsilon^2}}g(x(s-h))ds),
\end{equation}
For $C_{\mu}(\R)$-solutions, this equation takes the form $ x =
(\mathcal{I}_{\varepsilon}\circ \mathcal{G})x$.
\begin{theorem} \label{mrr}  Assume {\rm \bf(H)}, {\rm \bf(G)}. Let  $\psi $ be
the positive heteroclinic  from Lemma \ref{psi}.  Then for every
$\mu \not= \Re\lambda_j, \ \mu \in [0, \lambda),$ there are open
balls $\mathcal{E}_{\mu}= (-\varepsilon_\mu,\varepsilon_\mu),$
$\mathcal{V}_{\mu}\subset \R^{d(\mu)}$, and continuous family of
heteroclinics $\psi_{\varepsilon,v}: \mathcal{E}_{\mu}\times
\mathcal{V}_\mu \to C_{\mu}(\R)$ of Eq. (\ref{twe0}) such that
$\psi_{0,0} = \psi$. For each $\tilde{\varepsilon} \in
\mathcal{E}_{\mu}$, the subset $\{\psi_{\tilde{\varepsilon},v}: v
\in \mathcal{V}_\mu\} \subset C_{\mu}(\R)$ is $C^1-$manifold of
dimension $d(\mu)$. Moreover, there exists a
$C_\mu(\R)-$neighborhood $\mathcal{U}$ of $\psi$ and
$\varepsilon_1
> 0$ such that every solution $\psi_\varepsilon \in \mathcal{U}, \
|\varepsilon|< \varepsilon_1,$ of Eq. (\ref{twe0}) satisfies
$\psi_\varepsilon = \psi_{\varepsilon,v}$ for some $v \in
\mathcal{V}_\mu$. Finally,  given a closed subinterval $\mathcal{S}
\subset [0, \lambda)\setminus \{\Re\lambda_j\}$, we can choose open
sets $\mathcal{E}_{\mu}, \mathcal{V}_{\mu}$ to be constant on
$\mathcal{S}$.
\end{theorem}
\begin{proof} Set $R_{\mu} = (-1/\sqrt{\mu},
1/\sqrt{\mu})$ and then define $F:R_{\mu} \times C_{\mu}(\R)  \to
C_{\mu}(\R)$ by $F(\varepsilon,\phi) =\psi +\phi -
(\mathcal{I}_{\varepsilon}\circ \mathcal{G})(\psi + \phi).$ We
have that  $F(0,0)=0$. Furthermore, Lemmas \ref {8} and \ref{7}
imply that $F\in C(R_{\mu} \times C_{\mu}(\R),C_{\mu}(\R))$ and
$F_\phi(\varepsilon, \phi)$ is continuous in a neighborhood of
$(0,0)$. Set
$$
L:=F_{\phi}(0,0)=I-\mathcal{N}, \ V: ={\rm Ker}L, \
r(\varepsilon,\phi): = F(\varepsilon, \phi)-L\phi.
$$
Then $r_{\phi}(0,0)=F_{\phi}(0, 0)-L=0$. By Lemma \ref{6}, we have
that $\dim V<\infty$ and that $L$ is surjective. Thus $V$ has a
topological complement $W$ in $C_{\mu}(\R)$ so that
$C_{\mu}(\R)=V\oplus W$ and any $\phi\in C_{\mu}(\R)$ can be written
in the form $\phi=v+w$, $v \in V$ and $w \in W$. Recalling that
$Lv=0$ we get $ F(\varepsilon, \phi)=Lw +r(\varepsilon,v +w). $ This
suggests the following definition:
$$
\Phi(\varepsilon,v, w):=L|_Ww+r(\varepsilon, v+w),
$$
where $\Phi_{w}(0, 0, 0)= L|_W$ is the restriction of $L$ to $W$. Is
clear that $\Phi \in C(R_{\mu}\times V\times W, C_{\mu}(\R))$ and
$\Phi_{w}(\varepsilon, v, w)=L|_W+r_{\phi}(\varepsilon, v + w)$ is
continuous in a neighborhood of $(0,0,0)$. Since $L|_W: W \to
C_{\mu}(\R)$ is bijective we have that $(L|_{W})^{-1}$ is continuous
from $C_{\mu}(\R)$ to $W$. As a consequence, we can apply the
Implicit Function Theorem (e.g. see Theorem 2.3(i) in Ambrosetti \&
Prodi (1993)) to
$$\Phi(\varepsilon,v, w)=L|_Ww+r(\varepsilon, v+w)=0, \quad \Phi(0,0,
0)=0.$$  In this way, we find neighborhoods of $0$,
$\mathcal{E}_{\mu}\subset R_{\mu}$, $\mathcal V_{\mu}\subset V$ and
$\mathcal W_\mu\subset W$ and a continuous map $\gamma\in
C^{1}_v(\mathcal{E}_{\mu}\times \mathcal V_{\mu}, \mathcal W_\mu)$,
such that $\Phi(\varepsilon,v ,\gamma(\varepsilon,v))=0$ for all
$(\varepsilon,v) \in \mathcal{E}_{\mu}\times \mathcal V_{\mu}$.
Moreover, without restricting the generality, we can suppose that
$\Phi(\varepsilon,v, w)= 0$ with $(\varepsilon,v, w) \in
\mathcal{E}_{\mu}\times \mathcal V_{\mu} \times \mathcal W_\mu$
implies $w= \gamma(\varepsilon,v)$ (e.g. see Theorem 2.3(ii) in
Ambrosetti \& Prodi (1993)).

Hence, the continuous family $\psi_{\varepsilon,v}=\psi+
v+\gamma(\varepsilon,v): \mathcal{E}_{\mu}\times \mathcal V_{\mu}
\to C_{\mu}(\R)$ contains all  solutions of Eq. (\ref{twe0}) from
small neighborhoods of $\psi$, with  $\psi_{0,0} = \psi$. Since
$\gamma_v(0,0) =0$ and $\gamma_v (\varepsilon,v)$ is continuous
for each fixed $\varepsilon \in \mathcal{E}_{\mu}$, we conclude
that $\{\psi_{\varepsilon,v}: v \in \mathcal{V}_\mu\} \subset
C_{\mu}(\R)$ is $C^1-$smooth manifold of dimension $d(\mu)$.
Notice that (\ref{psea}) implies that $g(\psi_{\varepsilon,
v}(+\infty))=\psi_{\varepsilon, v}(+\infty)$. Thus
$\psi_{\varepsilon, v}(+\infty) = \psi_{0,0}(+\infty)= \kappa$, so
that $\{\psi_{\varepsilon,v}\}$ are heteroclinic solutions of
(\ref{twe0}).

Finally, the last conclusion of the theorem follows from the simple
observations that (a) the sets $\mathcal{E}_{\mu}, \mathcal V_{\mu},
\mathcal W_\mu$ are non-increasing in $\mu$ and (b) the function
$d(t)$ is piece-wise constant, with discontinuities at
$\{\Re\lambda_j\}\cap [0, \lambda)$.
 \end{proof}
\section{Asymptotic formulae}
Throughout this section, we denote by $\beta, \gamma, \eta, b,
C,C_j,C_*, \dots$ some positive constants that are independent of
the parameters $\varepsilon \in \Lambda_j: =
(-\varepsilon_j,\varepsilon_j), \ v \in \Omega$, where $1
>\varepsilon_0> \varepsilon_1 > \dots > \varepsilon_* > 0,$ and $\Omega \subset
\R^q$. We also assume that $h >0, p > 1$.
\begin{lemma}\label{111}
 Let continuous $x_{\varepsilon,v}(\cdot),
f_{\varepsilon,v}(\cdot):\Lambda_0 \times \Omega \times \R \to \R$
satisfy
\begin{equation} \label{t}
\varepsilon^2 x''(t)+ x'(t)-x(t)+px(t+h)=f_{\varepsilon,v}(t),\quad
t \in \R.
\end{equation}
\noindent Suppose further that $\displaystyle {\sup_{t\leq
0}[|x_{\varepsilon,v}(t)|+ |f_{\varepsilon,v}(t)|]}\leq C,$ $\
|x_{\varepsilon,v}(t)|\leq Ce^{-\gamma t},\ t\geq 0,$ and that $\
|f_{\varepsilon,v}(t)|\leq Ce^{-b t},\ t\geq 0,$ $(\varepsilon,v)
\in \Lambda_0 \times \Omega$. Then, given $\sigma \in (0,b)$, it
holds
$$x_{\varepsilon,v}(t)=z_{\varepsilon,v}(t) + w_{\varepsilon,v}(t),\ t \in \R,$$
where, with some continuous and bounded $B_j:
(-\varepsilon_*,\varepsilon_*) \times \Omega \to \C$,
$$z_{\varepsilon,v}(t)= \sum_{\gamma \leq \Re
\lambda_j(\varepsilon)<
b-\sigma}B_j(\varepsilon,v)e^{-\lambda_j(\varepsilon)t}$$ is a
finite sum of eigensolutions of (\ref{t}) associated to the roots
$\lambda_j(\varepsilon)\in \{\gamma \leq \Re
\lambda_j(\varepsilon)< b-\sigma\}$ of (\ref{char2}) and
$|w_{\varepsilon,v}(t)|\leq C_*e^{-(b-\sigma)t}, \ t \geq 0,$
$(\varepsilon,v) \in (-\varepsilon_*,\varepsilon_*)\times \Omega$.
\end{lemma}

\begin{proof}
Applying the Laplace transform $\mathcal{L}$ to equation
(\ref{t}), we obtain
$$
\chi(z,\varepsilon)\tilde{x}_{\varepsilon,v}(z)=\tilde{f}_{\varepsilon,v}(z)+r_{\varepsilon,v}(z),
$$
where $\chi(z,\varepsilon)=\varepsilon^2 z^2+z-1+p\exp(z h), \
\tilde{x}_{\varepsilon,v}=\mathcal{L}\{x_{\varepsilon,v}\}, \
\tilde{f}_{\varepsilon,v}=\mathcal{L}\{f_{\varepsilon,v}\},$ and
$$r_{\varepsilon,v}(z)=\varepsilon^2(x_{\varepsilon,v}'(0)+zx_{\varepsilon,v}(0))
+x_{\varepsilon,v}(0)+pe^{zh}\int_{0}^h
e^{-zu}x_{\varepsilon,v}(u)du.$$ Since $x_{\varepsilon,v}e^{\gamma
t}$ is bounded, $\tilde{x}_{\varepsilon,v}$ is holomorphic in the
open half-plane $\{\Re z>-\gamma\}$. Similarly,
$\tilde{f}_{\varepsilon,v}$ is holomorphic in $\{\Re z>-b\}$.
Since $r_{\varepsilon,v}$ is entire, the function
$$H_{\varepsilon,v}(z):=(\tilde{f}_{\varepsilon,v}(z)+r_{\varepsilon,v}(z))/\chi(z,\varepsilon)$$
is meromorphic in $\Re z>-b$, with only finitely many poles there.

\noindent{\it \underline{Step I.}} We claim that there are
$\sigma' \in (0,\sigma), \ \varepsilon_1>0,$ such that
$|H_{\varepsilon,v}(z)|\leq C_1/|z|,$ if $ \Re z = -b + \sigma'$,
$(\varepsilon, v) \in \Lambda_1 \times \Omega$. Indeed, take
$\sigma' \in (0,\sigma)$ such that the line $\Re z=-b+\sigma'$
does not contain any eigenvalue $-\lambda_j(\varepsilon), \
\varepsilon\in\overline{\Lambda}_1,$ and $1-b+\sigma' \not= 0$. We
have
\begin{equation*}
|\tilde{f}_{\varepsilon,v}(z)|\leq \int_0^{+\infty} e^{-\Re
zt}|f_{\varepsilon,v}(t)|dt\leq C\int_0^{+\infty} e^{-\Re
zt}e^{-bt}dt\leq \frac{C}{\sigma'}, \ \Re z \geq -b +\sigma';
\end{equation*}
\begin{equation*}
|r_{\varepsilon,v}(z)|\leq
\varepsilon^2(|x_{\varepsilon,v}'(0)|+|z||x_{\varepsilon,v}(0)|)+|x_{\varepsilon,v}(0)|+
pe^{\Re z h}\displaystyle\int_{0}^{h} e^{-\Re
zu}|x_{\varepsilon,v}(u)|du.
\end{equation*}
As a bounded solution of (\ref{t}), $x_{\varepsilon,v}$ should
satisfy,  for all $t \in \R, $
\begin{equation} \label{t2} \hspace{+3mm}
x_{\varepsilon,v}(t)=\frac{1}{\sqrt{1+4\varepsilon^2}}\left(
\int_{-\infty}^te^{\bar{\lambda}(t-s)}G_{\varepsilon,v}(s)ds
+\int_t^{+\infty}e^{\bar{\mu}(t-s)}G_{\varepsilon,v}(s)ds\right),
\end{equation}
where $\bar{\lambda}<0<\bar{\mu}$ are the roots of  $\varepsilon^2
z^2+z-1=0$ and
$G_{\varepsilon,v}(t):=px_{\varepsilon,v}(t+h)-f_{\varepsilon,v}(t).$
Differentiating (\ref{t2}), we obtain
\begin{equation} \label{t3}
 x_{\varepsilon,v}'(t)=\frac{1}{\sqrt{1+4\varepsilon^2}}
\left(\bar{\lambda}
\int_{-\infty}^te^{\bar{\lambda}(t-s)}G_{\varepsilon,v}(s)ds
+\bar{\mu}\int_t^{+\infty}e^{\bar{\mu}(t-s)}G_{\varepsilon,v}(s)ds\right),
\end{equation}
so that
$$\begin{displaystyle}|x_{\varepsilon,v}'(0)|\leq \frac{\bar{\mu}}{\sqrt{1+4\varepsilon^2}}\int_0^{+\infty}
e^{-\bar{\mu} s}|G_{\varepsilon,v}(s)|ds
+\frac{|\bar{\lambda}|}{\sqrt{1+4\varepsilon^2}}\int^0_{-\infty}
e^{-\bar{\lambda} s}|G_{\varepsilon,v}(s)|ds  \leq
\end{displaystyle}$$
$$ \begin{displaystyle}(p+1)C\left(\int_0^{+\infty} \bar{\mu}e^{-\bar{\mu} s}ds
+|\bar{\lambda}|\int^0_{-\infty} e^{-\bar{\lambda} s}ds\right) =
2C(p+1).
\end{displaystyle}
$$
Fix $k > -b +\sigma'$ and consider the vertical strip
$\Sigma_k:=\{-b+\sigma' \leq\Re z\leq k\}$, then
$$\begin{displaystyle} pe^{\Re zh}
\int_{0}^h|e^{-zu}x_{\varepsilon,v}(u)|du \leq
Cpe^{kh}\int_{0}^he^{b u}du:={C_3}, \ z \in \Sigma_k,
\end{displaystyle}$$
so that $\begin{displaystyle}|r_{\varepsilon,v}(z)|\leq
C_4(1+\varepsilon^2|z|)\end{displaystyle}, \  z \in \Sigma_k$.

\noindent Set $b(z)=-1+pe^{zh}$, then $|b(z)|\leq 1+pe^{kh} :=
\beta, z \in \Sigma_k,$ and
\begin{eqnarray}\label{H} \hspace{-5mm} |z||H_{\varepsilon,v}(z)|\leq
\dfrac{C_5(|z|+\varepsilon^2|z|^2)}{|\varepsilon^2 z^2+z+b(z)|}, \ \
z \in \Sigma_k.
\end{eqnarray}
Now,  set $y_0=\eta \beta$ for some $\eta> 2$ satisfying $\eta^2
\geq 2\beta^{-1}\sqrt{\eta^2\beta^2+b^2}$ and $\eta \beta >
b-\sigma'$. For all $z$ such that $\Re z=-b+{\sigma'}$, and $|\Im z
|\geq y_0$, we have
$$|\varepsilon z^2+z|=|z||\varepsilon^2 z+1|\geq y_0|\varepsilon^2
z+1|\geq \frac{y^2_0}{\sqrt{y_0^2+(b- \sigma')^2}} \geq 2\beta.$$
Thus $|\varepsilon^2 z^2+z +b(z)|\geq |\varepsilon^2 z^2+z|
-|b(z)|\geq |\varepsilon^2 z^2+z| -\beta\geq {|\varepsilon^2
z^2+z|}/{2}$, so that
\begin{equation}\label{ips}
\dfrac{(|z|+\varepsilon^2|z|^2)}{|\varepsilon^2 z^2+z
+b(z)|}\leq2\dfrac{1+\varepsilon^2|z|}{|\varepsilon^2 z+1|}\leq \eta
+\displaystyle\sup_{\Re z=-b +\sigma'}\dfrac{2|\varepsilon^2
z|}{|\varepsilon^2 z+1|}\leq 2\eta,
\end{equation}
for all
 $|\Im z|\geq
y_0$, $\Re z=-b+{\sigma'}$ and $\varepsilon \in \Lambda_1$.

\noindent Finally, for all $(z,\varepsilon) \in \{z: \Re
z=-b+{\sigma'}, |\Im z|\leq y_0\} \times \overline{\Lambda}_1$, we
have that
$$\dfrac{|z|+\varepsilon|z|^2}{|\varepsilon z^2+z +b(z)|}\leq
C_6.$$ Combining this inequality with (\ref{H}), (\ref{ips}), we
prove the main assertion of Step I.

\noindent{\it \underline{Step II.}} \  Taking $k > 0$, in virtue of
(\ref{H}) we can use the inversion formula
\begin{equation} \label{t4}x_{\varepsilon,v}(t)=\frac{1}{2\pi i}\int_{k-\infty
i}^{k+\infty i}e^{zt}\tilde{x}_{\varepsilon,v}(z)dz=\frac{1}{2\pi
i}\int_{k-\infty i}^{k+\infty i}e^{zt}H_{\varepsilon,v}(z)dz, \
t\geq 0.
\end{equation}
By Lemma \ref{lemma}, $H_{\varepsilon,v}(z)$ has only finitely
many poles in the strip $-b<\Re z\leq -\gamma$. Also,
$H_{\varepsilon,v}(z)\to 0$ uniformly in the strip $-b+\sigma'
\leq\Re z\leq k$, as $|\Im z|\to \infty$, and
$H_{\varepsilon,v}(-b+\sigma'+i \cdot)\in L_2$. Thus, we may shift
the path of integration in (\ref{t4}) to the left, to the line
$\Re z=-b+\sigma'$, and obtain $
x_{\varepsilon,v}(t)=z_{\varepsilon,v}(t)+w_{\varepsilon,v}(t), $
where
\begin{equation*}
z_{\varepsilon,v}(t)=\displaystyle \sum_{\gamma \leq \Re
\lambda_j(\varepsilon)< b-\sigma'}{\rm
Res}_{-\lambda_j(\varepsilon)}e^{zt}H_{\varepsilon,v}(z),\
w_{\varepsilon,v}(t)=\frac{1}{2\pi
i}\displaystyle\int\limits_{-b+\sigma'-\infty\cdot
i}^{-b+\sigma'+\infty\cdot i}e^{zt}H_{\varepsilon,v}(z)dz.
\end{equation*} By Lemma \ref{lemma}, the roots of
equation $\chi(z,\varepsilon)=0$ are simple for all small
$\varepsilon$. Hence
\begin{eqnarray*}
z_{\varepsilon,v}(t) =\sum_{\gamma \leq \Re \lambda_j(\varepsilon)<
b-\sigma'} e^{-\lambda_j(\varepsilon)t}B_j(\varepsilon,v), \ {\rm
with}\ B_j(\varepsilon,v)=
\dfrac{\tilde{f}_{\varepsilon,v}(-\lambda_j(\varepsilon))+
r_{\varepsilon,v}(-\lambda_j(\varepsilon))}{\chi'(-\lambda_j(\varepsilon),\varepsilon)}.
\end{eqnarray*}
It is easy to check that $B_j(\varepsilon,v)$ is continuous on its
domain of definition (observe here that the continuity of
 $x_{\varepsilon,v}'(0)$ follows from (\ref{t3})). Take $j$ such that $-b+ \sigma'<-\Re \lambda_j(\varepsilon)\leq
-\gamma$, then
$\begin{displaystyle}|r_{\varepsilon,v}(-\lambda_j(\varepsilon))|\leq
C_4(\varepsilon^2|\lambda_j(\varepsilon)|+1)\leq
C_4(\max_{j,\varepsilon}|\lambda_j(\varepsilon)|+ 1):=
C_7\end{displaystyle}$. In addition, if $\varepsilon \to 0$ then
 \begin{eqnarray*}0<|\chi'(-\lambda_j(\varepsilon),\varepsilon)|&=&|-2\varepsilon^2
 \lambda_j(\varepsilon)+1+phe^{-\lambda_j(\varepsilon)
h}| \to |1+phe^{-\lambda_jh}|\neq 0.\end{eqnarray*}
\begin{eqnarray*}\hspace{-5mm}
{ \rm Hence}, \  |B_j(\varepsilon,v)|&\leq&
\dfrac{|\tilde{f}_{\varepsilon,v}(-\lambda_j(\varepsilon))|+
|r_{\varepsilon,v}(-\lambda_j(\varepsilon))|}{|\chi'(-\lambda_j(\varepsilon),\varepsilon)|}\leq
 \dfrac{C/\sigma'+
C_7}{\displaystyle\min_{j,\varepsilon}|\chi'(-\lambda_j(\varepsilon),\varepsilon)|}\leq
C_{8}
\end{eqnarray*} if $\varepsilon \in \Lambda_{2}$, for some small $\varepsilon_{2}>0$  and $v \in \Omega$.

\noindent{\it \underline{Step III.}} \ Consider
$u_{\varepsilon,v}(t)=e^{(b-\sigma')t}w_{\varepsilon,v}(t)$ and
$v_{\varepsilon,v}(t)=e^{(b-\sigma)t}w_{\varepsilon,v}(t)$. We
have
$$
u_{\varepsilon,v}(t)=\frac{1}{2\pi
i}\displaystyle\int_{-b+\sigma'-\infty\cdot
i}^{-b+\sigma'+\infty\cdot
i}e^{(s+b-\sigma')t}H_{\varepsilon,v}(s)ds
=\frac{1}{2\pi}\displaystyle\int_{-\infty }^{+\infty}e^{i\xi
t}H_{\varepsilon,v}(-b+\sigma'+i\xi)d\xi.
$$
By Plancherel theorem,
$$\|u_{\varepsilon,v}\|_2=\frac{1}{2\pi}\|H_{\varepsilon,v}(-b+\sigma'+i\cdot)\|_2
\leq
 \frac{C_{1}}{2\sqrt{\pi(b-\sigma')}}.$$
Hence,
$v_{\varepsilon,v}(t)=e^{-(\sigma-\sigma')t}u_{\varepsilon,v}(t)$
is integrable on $[0,+\infty)$, and by the Cauchy-Schwarz
inequality
$$
\|v_{\varepsilon,v}\|_1\leq
\frac{\|u_{\varepsilon,v}\|_2}{\sqrt{2(\sigma-\sigma')}}\leq
\frac{C_{1}}{2\sqrt{2\pi(b-\sigma')(\sigma-\sigma')}}.
$$
{\it \underline{Step IV.}} \   We claim  that there exist real
numbers $C_9
>0$ and $\varepsilon_3>0$ such that $|w_{\varepsilon,v}(t)|\leq
C_9e^{-(b-\sigma)t}, t\geq 0,$ for all $(\varepsilon, v) \in
\Lambda_{3}\times \Omega$.  In order to prove this,  it suffices
to show that $v_{\varepsilon,v}$ is uniformly bounded for small
$\varepsilon \in \Lambda_{3}$. Since
\begin{equation*} \label{}
\varepsilon^2
w_{\varepsilon,v}''(t)+w_{\varepsilon,v}'(t)-w_{\varepsilon,v}(t)+
pw_{\varepsilon,v}(t+h)=f_{\varepsilon,v}(t),\quad t\in \R,
\end{equation*}
we find that
$v_{\varepsilon,v}(t)=e^{(b-\sigma)t}w_{\varepsilon,v}(t)$ satisfies
\begin{eqnarray*} \label{}
\varepsilon^2
v_{\varepsilon,v}''(t)+(1-2\varepsilon^2(b-\sigma))v_{\varepsilon,v}'(t)=P_{\varepsilon,v}(t),
\end{eqnarray*}
where $\alpha=1-2\varepsilon^2(b-\sigma) > 0$ and
$P_{\varepsilon,v} \in L_1[0,+\infty)$ is defined by
$$P_{\varepsilon,v}(t)=e^{(b-\sigma)t}f_{\varepsilon,v}(t)+(1+(b-\sigma)-\varepsilon^2(b-\sigma)^2)v_{\varepsilon,v}(t)
-pe^{-(b-\sigma)h}v_{\varepsilon,v}(t+h).$$ The variation of
constants formula yields
\begin{equation}\label{li}
v'_{\varepsilon,v}(t)=e^{-\frac{\alpha}{\varepsilon^2}t}\left(v'_{\varepsilon,v}(0)+\frac{1}{\varepsilon^2}
\int_0^te^{\frac{\alpha}{\varepsilon^2}s}P_{\varepsilon,v}(s)ds\right),
\ \varepsilon \not= 0.
\end{equation}
A direct integration of (\ref{li}) gives
\begin{equation*}v_{\varepsilon,v}(t)=v_{\varepsilon,v}(0)+\frac{\varepsilon^2}{\alpha}v'_{\varepsilon,v}(0)
(1-e^{-\frac{\alpha}{\varepsilon^2}t})+\frac{1}{\varepsilon^2}\int_0^t\int_0^ue^{\frac{\alpha}{\varepsilon^2}(s-u)}P_{\varepsilon,v}(s)dsdu.
\end{equation*}
After changing the order of integration in the iterated integral,
we get
$$\frac{1}{\varepsilon^2}\left|\int_0^t\int_s^te^{\frac{\alpha}{\varepsilon^2}(s-u)}P_{\varepsilon,v}(s)duds\right|
=\frac{1}{\alpha}\left|\int_0^tP_{\varepsilon,v}(s)(1-e^{\frac{\alpha}{\varepsilon^2}(s-t)})ds\right|
\leq \frac{1}{\alpha}\int_0^t|P_{\varepsilon,v}(s)|ds.
$$
Additionally, recalling Step II, we find  that
$|v'_{\varepsilon,v}(0)|\leq
(b-\sigma)|w_{\varepsilon,v}(0)|+|w_{\varepsilon,v}'(0)| \leq$
$$\begin{displaystyle}
\leq
(b-\sigma)(|x_{\varepsilon,v}(0)|+|z_{\varepsilon,v}(0)|)+|x_{\varepsilon,v}'(0)|+|z_{\varepsilon,v}'(0)|
< C_{10}.
\end{displaystyle}$$
As a consequence, for all small $\varepsilon$ and $v \in \Omega$,
we have that $$\begin{displaystyle}|v_{\varepsilon,v}(t)|\leq
|v_{\varepsilon,v}(0)|+\frac{\varepsilon^2}{\alpha}C_{10}(
1+e^{-\frac{\alpha}{\varepsilon^2}t})+\frac{1}{\alpha}\int_0^{+\infty}|P_{\varepsilon,v}(s)|ds
\leq C_{11}, \ t \geq 0.
\end{displaystyle}$$

Finally, since $w_{\varepsilon,v}(t)=v_{\varepsilon,v}(t)
e^{-(b-\sigma)t}$, Lemma \ref{111} is proved.
\end{proof}
\begin{theorem} \label{mrr13}  In Theorem \ref{mrr}, take $\mu = \lambda - \delta$, with small
$\delta >0$. Assume that $\psi $ is the positive heteroclinic of
(\ref{17a}) normalized by $\psi(t) = \exp (\lambda t) +
O(\exp((2\lambda- \delta)t)),$ $ \ t \to - \infty$. Then we can
choose  a neighborhood $\mathcal{U} \subset C_{\mu}(\R)$ of $\psi$
and  a neighborhood $\mathcal{E}_{\mu}^*\times  \mathcal{V}_\mu^*$
of $0 \in \R^2$ in such a way that $\psi_{\varepsilon,v} \in
\mathcal{U}, \ (\varepsilon,v) \in \mathcal{E}_{\mu}^* \times
\mathcal{V}_\mu^*,$ is positive and unique in $\mathcal{U}$ (up to
translations in $t$) for every fixed $\varepsilon$. Moreover,
$\psi_{\varepsilon,v}(t-t_0) = \exp (\lambda_1(\varepsilon) t) +
O(\exp(1.99\mu t))$ at $t \to - \infty$ for some $t_0 =
t_0(\varepsilon,v)\in \R$.
\end{theorem}
\begin{proof} First, we take $\mathcal{V}_\mu$, $\mathcal{E}_\mu \subset (-\varepsilon_1,\varepsilon_1), \ \mathcal{U}$
as in Theorem \ref{mrr}. It follows from Lemma \ref{lemma} and
Theorem \ref{mrr} that $\mathcal{V}_\mu\subset \R$ and that we can
choose positive $\delta$ and $\mathcal{E}_{\mu}$ such that $\Re
\lambda_j(\varepsilon) < \mu < \lambda < \lambda_1(\varepsilon)\ <
1.99\mu < \lambda_\infty(\varepsilon)$ for all $\varepsilon\in
\mathcal{E}_{\mu}$.   If we set
$y_{\varepsilon,v}(t)=\psi_{\varepsilon,v}(-t)$, then
$y_{\varepsilon,v}$ satisfies (\ref{t}) where
$$|f_{\varepsilon,v}(t)| =
|g(y_{\varepsilon,v}(t+h))-g'(0)y_{\varepsilon,v}(t+h)| \leq
C_1e^{-2\mu t}, \ t \geq -h.$$ Lemma \ref{111} assures that there
are $\mathcal{V}_\mu' \subset \mathcal{V}_\mu$, $\mathcal{E}_\mu'
\subset \mathcal{E}_\mu$ such that
$$y_{\varepsilon,v}(t)=B(\varepsilon,v)e^{- \lambda_1(\varepsilon) t}
+ w_{\varepsilon,v}(t), \ (\varepsilon,v) \in \mathcal{E}_{\mu}'
\times \mathcal{V}_\mu'.$$ Here $B: \mathcal{E}_{\mu}' \times
\mathcal{V}_\mu' \to \R_+,$ $B(0,0)=1,$ is continuous and
$|w_{\varepsilon,v}(t)|\leq C_*e^{-1.99\mu t}, \ $ $ t \geq 0, $ for
some $C_* >0$.

Hence, there are $\mathcal{E}_{\mu}'' \times \mathcal{V}_\mu''$
and $T
> 0$ (independent of $\varepsilon, v$) such that
$y_{\varepsilon,v}(t)
> 0.5e^{- \lambda_1(\varepsilon) t},$ $t > T$, for all
$(\varepsilon,v) \in \mathcal{E}_{\mu}'' \times
\mathcal{V}_\mu''$. On the other side, $\lim_{(\varepsilon, v) \to
0}y_{\varepsilon,v}(t) =\psi(-t)$ uniformly on $\R$. In
consequence, since  $\psi$ is bounded from below by a positive
constant on $[-T,\infty)$,  we conclude that $y_{\varepsilon,v}$
is positive on $\R$, if $(\varepsilon,v)$ belongs to sufficiently
small neighborhood $\mathcal{E}_{\mu}^* \times \mathcal{V}_\mu^*
\subset \mathcal{E}_{\mu}'' \times \mathcal{V}_\mu''$  of the
origin. Without the loss of the generality, we can assume
additionally that $\psi_{\varepsilon,v} \in \mathcal{U}$ for all
$(\varepsilon,v) \in \mathcal{E}_{\mu}^* \times
\mathcal{V}_\mu^*$.

Next, for every fixed $\varepsilon \in \mathcal{E}_{\mu}^*$, the
subset $\mathfrak{F}=\{\psi_{\varepsilon,v}: v \in
\mathcal{V}_\mu\} \subset C_{\mu}(\R)$ is homeomorphic to
$\mathcal{V}_\mu$.  On the other hand, for every  $n
>0$, the collection
$\mathfrak{P}_n=\{\psi_{\varepsilon,0}(t-s), \ s \in (-n,n)\}$ of
positive heteroclinics is a continuous 1-manifold in
$C_{\mu}(\R)$. Since $\psi_{\varepsilon,0} \in \mathfrak{F} \cap
\mathfrak{P}_{n}$ we obtain that $\{\psi_{\varepsilon,v}: v \in
\mathcal{V}_\mu^*\} \subset \mathfrak{P}_{\infty}$. In
consequence, $\psi_{\varepsilon,v}(t)$ is unique in $\mathcal{U}$
(up to shifts in $t$) for every fixed small $\varepsilon$.
\end{proof}

\begin{theorem}\label{111a} Set $ \mathcal{P} = \{(\varepsilon, v) \in \mathcal{E}_0\times
\mathcal{V}_0: \psi_{\varepsilon,v}(t) > 0, \ t \in \R\}$, where
$\mathcal{E}_0, \mathcal{V}_0$ are as in Theorem \ref{mrr}. Then
there exist a neighborhood $\mathcal{E}^*\times \mathcal{V}^*
\subset \mathcal{E}_0\times \mathcal{V}_0$ of $0$ and $C
>0$ such that, for all $(\varepsilon, v) \in \mathcal{P}^*:=\mathcal{P}
\cap (\mathcal{E}^*\times \mathcal{V}^*),$ we have that
\begin{equation}\label{good}
\psi_{\varepsilon,v}(t)=B(\varepsilon,v)e^{\lambda_1(\varepsilon)t}
+ w_{\varepsilon,v}(t),
\end{equation}
where $|w_{\varepsilon,v}(t)|\leq Ce^{1.99\lambda t}, t \leq 0,$ and
$B: \mathcal{E}^* \times \mathcal{V}^* \to (0,\infty)$ is
continuous.
\end{theorem}
\begin{proof} Let $\mathcal{E}' \subset \mathcal{E}_0$ be such that $\lambda_{\infty}(\varepsilon) > 3\lambda,$
for all $\varepsilon \in \mathcal{E}'$.  The last assertion of
Theorem \ref{mrr} implies that, for some $\gamma
> 0, \ C_1 > 0$,
\begin{equation}\label{expp}
\displaystyle {\sup_{t\geq 0}|\psi_{\varepsilon,v}(t)|}\leq C_1, \
|\psi_{\varepsilon,v}(t)|\leq C_1e^{\gamma t},\ t\leq 0.
\end{equation}
If we set $y_{\varepsilon,v}(t)=\psi_{\varepsilon,v}(-t)$, then
$y_{\varepsilon,v}$ satisfies (\ref{t}) where
$$|f_{\varepsilon,v}(t)| =
|g(y_{\varepsilon,v}(t+h))-g'(0)y_{\varepsilon,v}(t+h)| \leq
C_2e^{-2\gamma t}, \ t \geq -h.$$ Set $\Gamma = \sup \{\gamma > 0 \
{\rm such º\ that \ } (\ref{expp})\ {\rm holds \ for\ all \ }
(\varepsilon, v) \in \mathcal{P} \cap (\mathcal{E}'\times
\mathcal{V}_0)\}$.  Applying Lemma \ref{111}, we get
$$y_{\varepsilon,v}(t)=\sum_{0 < \lambda_j(\varepsilon)<
2\Gamma}B_j(\varepsilon,v)e^{- \lambda_j(\varepsilon) t} + \tilde
w_{\varepsilon,v}(t),$$ where $B_j: \mathcal{E}''  \times
\mathcal{V}_0 \to \C$ are continuous and $|\tilde
w_{\varepsilon,v}(t)|\leq C_3e^{-1.99\Gamma t}, \ $ $ t \geq 0, $ $
(\varepsilon, v) \in \mathcal{P} \cap (\mathcal{E}''\times
\mathcal{V}_0)$, for some $C_3
>0$ and open $\mathcal{E}'' \subset \mathcal{E}'$. Since $\Gamma >0$ is finite and $y_{\varepsilon,v}(t) > 0,$
we obtain
$$
\sum_{0 < \lambda_j(\varepsilon)< 2\Gamma}B_j(\varepsilon,v)e^{
-\lambda_j(\varepsilon) t} = B(\varepsilon,v)e^{
-\lambda_1(\varepsilon) t},
$$
so that $\Gamma \geq  \lambda$, see Lemma \ref{lemma}. Next, due
to Lemma \ref{psi}, it holds that $B(0,0)>0$. Hence, $\Gamma =
\lambda$.
\end{proof}
\begin{corollary}\label{111b} Given $\delta \in (0,\lambda)$ and $(\varepsilon_j,v_j)  \in \mathcal{P}^*, j = 0,1,\dots$, the convergence
$$\psi_{\varepsilon_j,v_j} \stackrel{C_{0}(\R)}\longrightarrow
\psi_{\varepsilon_0,v_0} \qquad {\rm implies} \qquad
\psi_{\varepsilon_j,v_j}
\stackrel{C_{\lambda-\delta}(\R)}\longrightarrow
\psi_{\varepsilon_0,v_0}. $$
\end{corollary}
\begin{proof} By the contrary, suppose that there are a
sequence $\{\psi_{\varepsilon_j,v_j},  (\varepsilon_j,v_j) \in
\mathcal{P}^*\}_{j \geq 0}$ and $\eta
> 0$ such that
$$
\lim_j|\psi_{\varepsilon_j, v_j} - \psi_{\varepsilon_0, v_0}|_0 =0,
\ |\psi_{\varepsilon_j, v_j} - \psi_{\varepsilon_0,
v_0}|_{\lambda-\delta}
> \eta, \ j =1,2, \dots
$$
It follows from (\ref{good}) that there exist $C > 0$ and $T < 0$
such that $$\psi_{\varepsilon_j, v_j}(t)e^{-(\lambda-\delta)t} \leq
Ce^{\delta t} < \eta/4, \ j =0,1,2, \dots,  \ t \leq T. $$ Thus
$$
\sup_{s \leq T}\left[e^{-(\lambda - \delta)s}|\psi_{\varepsilon_j,
v_j}(s) - \psi_{\varepsilon_0, v_0}(s)|\right] \leq \eta/2, \ j
=1,2, \dots.
$$
Next, since $\psi_{\varepsilon_j, v_j}(t) \to \psi_{\varepsilon_0,
v_0}(t)$ uniformly on $\R$, we can find $j_*$ such that
$$
\sup_{s \in [T,0]}\left[e^{-(\lambda -
\delta)s}|\psi_{\varepsilon_j, v_j}(s) - \psi_{\varepsilon_0,
v_0}(s)|\right] \leq \frac{\eta}{2},\ \sup_{s \geq
0}|\psi_{\varepsilon_j, v_j}(s) - \psi_{\varepsilon_0, v_0}(s)| \leq
\frac{\eta}{2}, \ j \geq j_*.
$$
But all this means that $|\psi_{\varepsilon_j, v_j} -
\psi_{\varepsilon_0, v_0}|_{\lambda-\delta} \leq \eta/2$ for all $j
\geq j_*$, a contradiction.
\end{proof}
\section{Proof of Theorem \ref{mr}}
Everywhere below, all positive wavefronts $\phi$ will be
normalized by the conditions $\phi(0) =\zeta_1/2$ and $\phi(s) <
\zeta_1/2,$ $s < 0$, with $\zeta_1$ defined in {\rm \bf(G)}. Let
$\psi,$ $ \psi(0) = \zeta_1/2$, $\psi(s) < \zeta_1/2,$  $s < 0$,
be the positive heteroclinic of (\ref{17a}) given in Lemma
\ref{psi}.  By Theorem \ref{mrr13}, there exists a neighborhood
$(-\varepsilon_0,\varepsilon_0) \times \mathcal{U}\subset \R
\times C_{\lambda-\delta}(\R)$ of $(0,\psi)$ such that for every
fixed $\varepsilon \in (-\varepsilon_0,\varepsilon_0)$ there is a
unique normalized positive wavefront $\psi_{\varepsilon} \in
\mathcal{U}$. We claim that, if $\varepsilon$ is sufficiently
small, then this $\psi_{\varepsilon}$ will be the unique
normalized positive wavefront of Eq. (\ref{twe0}). Indeed, let us
suppose, for instance, that we can find a sequence $\varepsilon_j
\to 0$ and normalized positive wavefronts
$\phi_{\varepsilon_j}\not= \psi_{\varepsilon_j}$.
\begin{lemma} \label{16} Assume $\mathbf{(H)}$ and $\mathbf{(G)}$. Then
${\phi}_{\varepsilon_j} \to \psi$  uniformly on $\mathbb{R}$.
\end{lemma}
\begin{proof}
First, we prove the uniform convergence ${\phi}_{\varepsilon_j}
\to \psi$  on compact subsets of $\R$.  Since $g$ is a bounded
function, we obtain from (\ref{psea}) that
\begin{equation*}\label{pro}
 |{\phi}_{\varepsilon_j} '(t)|+ |{\phi}_{\varepsilon_j} (t)| \leq \frac{\max_{x \geq
0}g(x)}{\varepsilon^2 (\mu - \lambda)} +\max_{x \geq 0}g(x) \leq
2\zeta_2, \ j \in \N.
\end{equation*}  Hence, by the Ascoli-Arzel${\rm
\grave{a}}$ theorem combined with the diagonal method,
$\{{\phi}_{\varepsilon_j}\}$ is precompact in $C(\R,\R)$. Thus,
every $\{{\phi}_{\varepsilon_{j_k}}\}$  has a subsequence
converging in $C(\R,\R)$ to some continuous positive bounded
function $\varphi(s)$ such that $\varphi'(s) \geq 0, s \leq 0,$
and $\varphi (0) = \zeta_1/2$. Making use of the Lebesgue's
dominated convergence theorem, we deduce from Eq. (\ref{psea})
that
$${\varphi}(t)=\int_{-\infty}^{t}e^{-(t-s)}g({\varphi}(s-h))ds.$$
Therefore ${\varphi}$ is a positive bounded solution of  Eq.
(\ref{17a}) and since the equilibrium $\kappa$ of Eq. (\ref{17a})
is globally attractive, it holds that ${\varphi}(+\infty)=\kappa.$
On the other hand, since $\varphi (-\infty) \leq \varphi (0) =
\zeta_1/2$, we have that ${\varphi}(-\infty)=0$. Hence, due to
Lemma \ref{psi}, we obtain that ${\varphi}(t) = \psi(t), \ t \in
\R$. Next, if ${\varphi}_{\varepsilon_n}\not\to {\psi}$ uniformly
on $\mathbb{R}$ then there exist a subsequence
$\{\varphi_{\varepsilon_{j_{n}}}\} \subset
\{{\varphi}_{\varepsilon_j}\}$ (for short, we will write again
$\{{\varphi}_{\varepsilon_j}\}$ instead of
$\{{\varphi}_{\varepsilon_{j_n}}\}$),   a sequence $\{S_j\}$ and
positive numbers $T,\delta < \kappa/6$ such that
$$|\psi(S_j)-{\varphi}_{\varepsilon_j}(S_j)|=2\delta, \ |\psi(t)|<0.25\delta, t \leq -T, \ |\psi(t)-\kappa|<0.25\delta, \  t \geq T.$$
Since ${\varphi}_{\varepsilon_n}$ converges uniformly on $[-2T,2T]$
to ${\psi}$, and ${\varphi}_{\varepsilon_n}, \ \psi$ are monotone
increasing on $(-\infty,0]$, we can suppose that
$|\psi(t)-{\varphi}_{\varepsilon_n}(t)|< \delta$ for all $t \in
(-\infty,2T]$ and $n\geq n_0$. In this way, $S_j \to +\infty$ and we
can suppose that
$$
|\psi(t)-{\varphi}_{\varepsilon_j}(t)|< 2\delta, \ t \in
(-\infty,S_j).
$$
Consider the  sequence $y_j(t)={\varphi}_{\varepsilon_j}(t+S_j)$
of heteroclinics to Eq. (\ref{twe0}). We have that
$|y_j(0)-\kappa|> 1.5 \delta$ and $|y_j(t)-\kappa|<3 \delta$ when
$t\in (T-S_j,0)$. Arguing as above, we find that $\{y_j\}$
contains a subsequence converging, on compact subsets of $\R$, to
some solution $y_*(t)$ of (\ref{17a}) satisfying
$|y_*(0)-\kappa|\geq 1.5\delta$ and $|y_*(t)-\kappa|\leq 3\delta
<\frac{\kappa}{2}$ for all $t<0$. Lemma \ref{13} implies that
$\inf_\R y_*(t)>0$. Since $y_*(0)\not= \kappa$, we have
established the existence of a non-constant positive bounded and
separated from $0$ solution to (\ref{17a}). This contradicts to
the global attractivity of $\kappa$.
\end{proof}
\begin{corollary}\label{b}
$\phi_{\varepsilon_j} \to \psi$ in $C_{\lambda-\delta}(\R)$.
\end{corollary}
\begin{proof} Since $\phi_{\varepsilon_j} \to \psi$ in $C_{0}(\R)$, we have that  $\phi_{\varepsilon_j} = \psi_{\varepsilon_j,v_j}$
for some $v_j \in \mathcal{V}_0$. Now we can apply Corollary
\ref{111b} to find that $\phi_{\varepsilon_j} \to \psi$ in
$C_{\lambda-\delta}(\R)$.
\end{proof}
Lastly, Theorem \ref{mrr13} and Corollary \ref{b} implies that
$\phi_{\varepsilon_j}= \psi_{\varepsilon_j}$, a contradiction which
completes the proof of Theorem \ref{mr}.

\section*{Acknowledgments}
The authors thank Teresa Faria for useful discussions. S. Trofimchuk
was partially supported by CONICYT (Chile) through PBCT program
ACT-05 and by the University of Talca, program ``Reticulados y
Ecuaciones".
 S. Trofimchuk and G. Valenzuela were supported by
FONDECYT (Chile) project 1071053.

\end{document}